\documentclass{amsart}

\usepackage{enumerate, amsmath, amsfonts, amssymb, amsthm, wasysym, graphics, graphicx, xcolor, url, hyperref, hypcap, a4wide}
\hypersetup{colorlinks=true, citecolor=darkblue, linkcolor=darkblue}
\usepackage[all]{xy}
\usepackage{tikz}\usetikzlibrary{trees,snakes,shapes,arrows,matrix,calc}
\graphicspath{{figures/}}

\title{Which nestohedra are removahedra?}

\thanks{Supported by the spanish MICINN grant MTM2011-22792 and the french ANR grant EGOS (12 JS02 002 01).}

\author{Vincent Pilaud}
\address{CNRS \& LIX, \'Ecole Polytechnique, Palaiseau}
\email{vincent.pilaud@lix.polytechnique.fr}
\urladdr{http://www.lix.polytechnique.fr/~pilaud/}

\newtheorem{theorem}{Theorem}
\newtheorem{corollary}[theorem]{Corollary}

\newtheorem{lemma}[theorem]{Lemma}
\newtheorem{definition}[theorem]{Definition}

\theoremstyle{definition}
\newtheorem{example}[theorem]{Example}
\newtheorem{remark}[theorem]{Remark}

\newcommand{\R}{\mathbb{R}} 
\newcommand{\N}{\mathbb{N}} 
\newcommand{\Z}{\mathbb{Z}} 
\newcommand{\I}{\mathbb{I}} 
\newcommand{\fS}{\mathfrak{S}} 
\newcommand{\HH}{\mathbb{H}} 
\renewcommand{\b}[1]{\mathbf{#1}} 

\newcommand{\set}[2]{\left\{ #1 \;\middle|\; #2 \right\}} 
\newcommand{\bigset}[2]{\big\{ #1 \;|\; #2 \big\}} 
\newcommand{\biggset}[2]{\bigg\{ #1 \;\bigg|\; #2 \bigg\}} 
\newcommand{\ssm}{\smallsetminus} 
\newcommand{\dotprod}[2]{\langle #1 | #2 \rangle} 
\newcommand{\symdif}{\triangle} 
\newcommand{\one}{{1\!\!1}} 
\newcommand{\eqdef}{\mbox{\,\raisebox{0.2ex}{\scriptsize\ensuremath{\mathrm:}}\ensuremath{=}\,}} 
\newcommand{\defeq}{\mbox{~\ensuremath{=}\raisebox{0.2ex}{\scriptsize\ensuremath{\mathrm:}} }} 
\newcommand{\simplex}{\triangle} 

\newcommand{\Perm}{\mathsf{Perm}} 
\newcommand{\Remo}{\mathsf{Remo}} 
\newcommand{\Defo}{\mathsf{Defo}} 
\newcommand{\Mink}{\mathsf{Mink}} 
\newcommand{\ground}{\mathsf{S}} 
\newcommand{\building}{\mathsf{B}} 
\newcommand{\summands}{\mathsf{C}} 
\newcommand{\nested}{\mathsf{N}} 
\newcommand{\nestedComplex}{\mathcal{N}} 
\newcommand{\nestedFan}{\mathcal{F}} 
\newcommand{\btree}{\mathsf{T}} 
\newcommand{\pathG}{\mathrm{P}} 
\newcommand{\cycle}{\mathrm{O}} 
\newcommand{\tree}{\mathrm{T}} 
\newcommand{\graphG}{\mathrm{G}} 
\newcommand{\descendants}{\mathrm{desc}} 
\newcommand{\bpath}{\pi} 
\newcommand{\bpaths}{\Pi} 
\newcommand{\point}{\mathbf{a}} 
\newcommand{\Hyp}{\b{H}^=} 
\newcommand{\HS}{\b{H}^\ge} 
\newcommand{\fan}{\mathcal{F}} 
\newcommand{\normalCone}{\mathbf{C}} 
\newcommand{\topmost}{\wedge} 

\DeclareMathOperator{\conv}{conv} 
\DeclareMathOperator{\cone}{cone} 

\newcommand{\fref}[1]{Figure~\ref{#1}} 
\newcommand{\ie}{\textit{i.e.}~} 
\newcommand{\eg}{\textit{e.g.}~} 
\newcommand{\aka}{\textit{aka.}~} 
\newcommand{\ex}[1]{^{\textrm{ex#1}}} 
\definecolor{darkblue}{rgb}{0,0,0.7} 
\newcommand{\darkblue}{\color{darkblue}} 
\newcommand{\defn}[1]{\emph{\darkblue #1}} 
\usepackage{todonotes}

\begin{document}

\begin{abstract}
A removahedron is a polytope obtained by deleting inequalities from the facet description of the classical permutahedron. Relevant examples range from the associahedra to the permutahedron itself, which raises the natural question to characterize which nestohedra can be realized as removahedra. In this note, we show that the nested complex of any connected building set closed under intersection can be realized as a removahedron. We present two different complementary proofs: one based on the building trees and the nested fan, and the other based on Minkowski sums of dilated faces of the standard simplex. In general, this closure condition is sufficient but not necessary to obtain removahedra. However, we show that it is also necessary to obtain removahedra from graphical building sets, and that it is equivalent to the corresponding graph being chordful (\ie any cycle induces a clique).

\medskip
\noindent
{\sc keywords.}
Building set, nested complex, nestohedron, graph associahedron, generalized permutahedron, removahedron.
\end{abstract}

\maketitle


\section{Introduction}

The \defn{permutahedron} is a classical polytope, obtained as the convex hull of all permutations of~$[n]$, which is closely related to various properties of the symmetric group. Relevant polytopes can be obtained from the permutahedron by:
\begin{enumerate}[(i)]
\item gliding its facets orthogonally to its normal vectors without passing any vertex; the resulting polytopes are called \defn{deformed permutahedra} and were studied by A.~Postnikov in~\cite{Postnikov};
\\[-.3cm]
\item deleting inequalities from its facet description; we call the resulting polytopes \defn{removahedra}.
\end{enumerate}
Our interest in deformed permutahedra and removahedra is motivated by certain realizations of the associahedron and their generalizations to graph associahedra and nestohedra. An \defn{associahedron} is a polytope whose $1$-skeleton realizes the rotation graph on binary trees with $n$ vertices. In~\cite{Loday}, J.-L.~Loday constructed a remarkable realization of the associahedron which happened to be a removahedron. Following the same line, C.~Hohlweg and C.~Lange~\cite{HohlwegLange} later constructed~$2^n$ removahedra realizing the associahedron. In a different direction, M.~Carr and S.~Devadoss~\cite{CarrDevadoss} defined \defn{graph associahedra}, which realize the clique complex of a compatibility relation on tubes (connected induced subgraphs) of a fixed graph. Extending these polytopes, A.~Postnikov~\cite{Postnikov} and independently by E.-M.~Feichtner and B.~Sturmfels~\cite{FeichtnerSturmfels} constructed \defn{nestohedra}, which realize the nested complex on a building set, see Section~\ref{subsec:nestedComplex} for definitions. The resulting polytopes are all deformed permutahedra, but not always removahedra.

In this note, we investigate which nestohedra can be realized as removahedra. We show that the nested complex of a connected building set closed under intersection can always be realized as a removahedron. For graph associahedra, this closure condition is equivalent to the underlying graph being chordful (\ie any cycle induces a clique). Conversely, we show that graph associahedra realizable as removahedra are precisely chordful graph associahedra. We develop two complementary approaches to these questions. The first one, based on building trees and the nested fan, describes the vertices of the resulting removahedra. The other one, based on Minkowski sums of dilated faces of the standard simplex, describes the dilation factors in the Minkowski decomposition of the resulting removahedra.


\section{Preliminaries}

\subsection{Permutahedra, deformed permutahedra, and removahedra}

All polytopes considered in this note are closely related to the braid arrangement and to the classical permutahedron. Therefore, we first recall the definition and basic properties of the permutahedron (see~\mbox{\cite[Lect.\,0]{Ziegler}}) and certain relevant deformations of it. We fix a finite ground set~$\ground$ and denote by~$\set{e_s}{s \in \ground}$ the canonical basis of~$\R^\ground$.

\begin{definition}
\label{def:permutahedron}
The \defn{permutahedron}~$\Perm(\ground)$ is the convex polytope obtained equivalently as
\begin{enumerate}[(i)]
\item either the convex hull of the vectors~$\sum_{s \in \ground} \sigma(s) e_s \in \R^\ground$ for all bijections~$\sigma : \ground \to [|\ground|]$,
\item or the intersection of the hyperplane~$\HH \eqdef \Hyp(\ground)$ with the half-spaces~$\HS(R)$ for~${\varnothing \ne R \subset \ground}$, where
\[
\Hyp(R) \eqdef \biggset{\b{x} \in \R^\ground}{\sum_{r \in R} x_r = \binom{|R|+1}{2}} \quad \text{and} \quad \HS(R) \eqdef \biggset{\b{x} \in \R^\ground}{\sum_{r \in R} x_r \ge \binom{|R|+1}{2}},
\]
\item or the Minkowski sum of all segments~$[e_r,e_s]$ for~$r \ne s \in \ground$.
\end{enumerate}
\end{definition}

The normal fan of the permutahedron is the fan defined by the \defn{braid arrangement} in~$\HH$, \ie the arrangement of the hyperplanes~$\set{\b{x} \in \HH}{x_r = x_s}$ for~$r \ne s \in \ground$. Its $k$-dimensional cones correspond to the surjections from~$\ground$ to~$[k+1]$, or equivalently to the ordered partitions of~$\ground$ into~$k+1$ parts. In this note, we are interested in the following deformations of the permutahedron~$\Perm(\ground)$. These polytopes were called \defn{generalized permutahedra} by A.~Postnikov~\cite{Postnikov, PostnikovReinerWilliams}, but we prefer the term \defn{deformed} to distinguish from the other natural generalization of the permutahedron to finite Coxeter groups.

\begin{definition}[\cite{Postnikov, PostnikovReinerWilliams}]
\label{def:deformedPermutahedron}
A \defn{deformed permutahedron} is a polytope whose normal fan coarsens that of the permutahedron. Equivalently~\cite{PostnikovReinerWilliams}, it is a polytope defined as
\[
\Defo(\b{z}) \eqdef \biggset{\b{x} \in \R^\ground}{\sum_{s \in \ground} x_s = z_\ground \text{ and } \sum_{r \in R} x_r \ge z_R \text{ for all } \varnothing \ne R \subset \ground},
\]
for some~$\b{z} \eqdef (z_R)_{R \subseteq \ground} \in (\R_{>0})^{2^\ground}$,  such that $z_R + z_{R'} \le z_{R \cup R'} + z_{R \cap R'}$ for any~$R, R' \subseteq \ground$.
\end{definition}

As the permutahedron itself, all deformed permutahedra can be decomposed as Minkowski sums and differences of dilates of faces of the standard simplex~\cite{ArdilaBenedettiDoker}. For our purposes, we only need here the following simpler fact, already observed in~\cite{Postnikov}.

\begin{remark}[\cite{Postnikov}]
For any~$S \subseteq \ground$, we consider the face~$\simplex_S \eqdef \conv \set{e_s}{s \in S}$ of the standard simplex~$\simplex_\ground$.
For any~$\b{y} \eqdef (y_S)_{S \subseteq \ground}$ where all~$y_S$ are non-negative real numbers, the Minkowski sum
\[
\Mink(\b{y}) \eqdef \sum_{S \subseteq \ground} y_S \simplex_S
\]
of dilated faces of the standard simplex is a deformed permutahedron~$\Defo(\b{z})$, and the values ${\b{z} = (z_R)_{R \subseteq \ground}}$ of the right hand sides of the inequality description of~$\Mink(\b{y}) = \Defo(\b{z})$ are given~by
\[
z_R = \sum_{S \subseteq R} y_S.
\]
\end{remark}

\begin{remark}
\label{rem:genericMinkowskiSum}
As the normal fan of a Minkowski sum is just the common refinement of the normal fans of its summands, and the normal fan is invariant by dilation, the combinatorics of the face lattice of the Minkowski sum~$\Mink(\b{y})$ only depends on the set~$\set{S \subseteq \ground}{y_S > 0}$ of non-vanishing dilation factors. When we want to deal with combinatorics only, we denote generically by~$\Mink[\summands]$ any Minkowski sum~$\Mink(\b{y})$ with dilation factors~$\b{y} = (y_S)_{S \subseteq \ground}$ such that~$\summands = \set{S \subseteq \ground}{y_S > 0}$.
\end{remark}

Among these deformed permutahedra, some are simpler than the others as all their facet defining inequalities are also facet defining inequalities of the classical permutahedron. In other words, they are obtained from the permutahedron by removing facets, which motivates the following name.

\begin{definition}
\label{def:removahedron}
A \defn{removahedron} is a polytope obtained by removing inequalities from the facet description of the permutahedron, \ie a polytope defined for some~$\building \subseteq 2^\ground$ by
\[
\Remo(\building) \eqdef \HH \cap \bigcap_{B \in \building} \HS(B) = \biggset{\b{x} \in \HH}{\sum_{s \in B} x_s \ge \binom{|B|+1}{2} \text{ for all } B \in \building}.
\]
\end{definition}

\subsection{Building set, nested complex, and nested fan}
\label{subsec:nestedComplex}

We now switch to building sets and their nested complexes. We only select from~\cite{CarrDevadoss, Postnikov, FeichtnerSturmfels, Zelevinsky} the definitions needed in this note. More details and motivation can be found therein.

\begin{definition}
\label{def:building}
A \defn{building set}~$\building$ on a ground set~$\ground$ is a collection of non-empty subsets of~$\ground$~such~that
\begin{enumerate}[(B1)]
\item if~$B,B' \in \building$ and~$B \cap B' \ne \varnothing$, then~$B \cup B' \in \building$, and
\item $\building$ contains all singletons~$\{s\}$ for~$s \in \ground$.
\end{enumerate} 
A building set is \defn{connected} if~$\ground$ is the unique maximal element. Moreover, we say that a building set~$\building$ is \defn{closed under intersection} if~$B, B'\in\building$ implies $B\cap B'\in \building \cup \{\varnothing\}$.
\end{definition}

All building sets in this manuscript are assumed to be connected and we will study the relation between removahedra and building sets closed under intersection. We first recall a general example of building sets, arising from connected subgraphs of a graph.

\begin{example}
\label{exm:graphicalBuildingSet}
Given a graph~$\graphG$ with vertex set~$\ground$, we denote by~$\building\graphG$ the \defn{graphical building set} on~$\graphG$, \ie the collection of all non-empty subsets of~$\ground$ which induce connected subgraphs of~$\graphG$. The maximal elements of~$\building\graphG$ are the vertex sets of the connected components of~$\graphG$, and we will therefore always assume that the graph~$\graphG$ is connected. We call a graph~$\graphG$ \defn{chordful} if any cycle of~$\graphG$ induces a clique. Observe in particular that every tree is chordful. The following statement describes the graphical building sets of chordful graphs.
\end{example}

\begin{lemma}
\label{lem:graphicalBuildingSet}
A (finite connected) graph~$\graphG$ is chordful if and only if its graphical building set is closed under intersection.
\end{lemma}

\begin{proof}
Assume that~$\graphG$ is chordful, and consider~$B,B' \in \building\graphG$ and $s,t \in B \cap B'$. As~$B$ and~$B'$ induce connected subgraphs of~$\graphG$, there exists paths~$P$ and~$P'$ between~$s$ and~$t$ in~$\graphG$ whose vertex sets are contained in~$B$ and~$B'$, respectively. The symmetric difference~$P \symdif P'$ of these paths is a collection of cycles. Since $\graphG$ is chordful, we can replace in each of these cycles the subpath of~$P$ (resp.~of~$P'$) by a chord. We thus obtain a path from~$s$ to~$t$ which belongs to~$B \cap B'$. It follows that~$B \cap B'$ induces a connected subgraph of~$G$ and thus that~$B \cap B' \in \building\graphG$.

Assume reciprocally that~$\graphG$ has a cycle~$(s_i)_{i \in \Z_\ell}$, with a missing chord~$s_xs_y$. Consider the subsets~$B \eqdef \set{s_i}{x \le i \le y}$ and~$B' \eqdef \set{s_i}{y \le i \le x}$, where the inequalities between labels in~$\Z_\ell$ have to be understood cyclically. Clearly,~$B,B' \in \building\graphG$ while~$B \cap B' = \{s_x,s_y\} \notin \building\graphG$.
\end{proof}

\begin{example}
\label{exm:buildingSets}
The following sets are building sets:
\begin{itemize}
\item $\building\ex{0} \eqdef 2^{[4]} \ssm \{\varnothing\}$ is the graphical building set over the complete graph~$K_4$,
\item $\building\ex{1} \eqdef 2^{[4]} \ssm \big\{\varnothing, \{1,3\} \! \big\}$ is the graphical building set over the graph~$K_4 \ssm \big\{ \! \{1,3\}\! \big\}$,
\item $\building\ex{2} \eqdef 2^{[4]} \ssm \big\{\varnothing, \{1,3\}, \{1,4\}, \{1,3,4\} \! \big\}$ is the graphical building set over~$K_4 \ssm \big\{ \! \{1,3\}, \{1,4\} \! \big\}$,
\item $\building\ex{3} \eqdef \big\{ \! \{1\}, \{2\}, \{3\}, \{4\}, \{5\}, \{1,2,3\}, \{1,3,4,5\}, \{1,2,3,4,5\} \! \big\}$ is not graphical,
\item $\building\ex{4} \eqdef \big\{ \! \{1\}, \{2\}, \{3\}, \{4\}, \{5\}, \{1,2,3,4\}, \{3,4,5\}, \{1,2,3,4,5\} \! \big\}$ is not graphical.
\end{itemize}
The two building sets~$\building\ex{0}$ and~$\building\ex{2}$ are closed under intersection, while the other three are not.
\end{example}

In this note, we focus on polytopal realizations of the nested complex of a building set, a simplicial complex defined below. Following~\cite{Zelevinsky}, we do not include~$\ground$ in the definition of $\building$-nested sets in order for the $\building$-nested complex to be a simplicial complex.

\begin{definition}
\label{def:nested}
A \defn{$\building$-nested set}~$\nested$ is a subset of~$\building \ssm \{\ground\}$ such that
\begin{enumerate}[(N1)]
\item for any~$N,N' \in \nested$, either~$N \subseteq N'$ or~$N' \subseteq N$ or~$N \cap N' = \varnothing$, and
\item for any~$k \ge 2$ pairwise disjoint sets~$N_1,\dots,N_k \in \nested$, the union~$N_1 \cup \dots \cup N_k$ is not in~$\building$.
\end{enumerate}
The \defn{$\building$-nested complex} is the simplicial complex~$\nestedComplex(\building)$ of all $\building$-nested sets.
\end{definition}

\begin{example}
For a graphical building set~$\building\graphG$, Conditions~\textit{(N1)} and~\textit{(N2)} in Definition~\ref{def:nested} can be replaced by the following: for any~$N,N' \in \nested$, either~$N \subseteq N'$ or~$N' \subseteq N$ or~$N \cup N' \notin \building\graphG$. In particular, the $\building\graphG$-nested complex is a clique complex: a simplex belongs to~$\building\graphG$ if and only if all its edges belong to~$\building\graphG$.
\end{example}

The $\building$-nested sets can be represented by the inclusion poset of their elements. Since we only consider connected building sets, the Hasse diagrams of these posets are always trees. In the next definition, we consider rooted trees whose vertices are labeled by subsets of~$\ground$. For any vertex~$v$ in a rooted tree~$\btree$, we call \defn{descendant set} of~$v$ in~$\btree$ the union~$\descendants(v, \btree)$ of the label sets of all descendants of~$v$ in~$\btree$, including the label set of the vertex~$v$ itself. The $\building$-nested sets are then in bijection with the following $\building$-trees.

\begin{definition}
\label{def:btrees}
A \defn{$\building$-tree} is a rooted tree whose label sets partition~$\ground$ and such that
\begin{enumerate}
\item for any vertex~$v$ of~$\btree$, the descendant set~$\descendants(v, \btree)$ belongs to~$\building$,
\item for any~$k \ge 2$ incomparable vertices~$v_1, \dots, v_k \in \btree$, the union~$\bigcup_{i \in [k]} \descendants(v_i, \btree)$ is not in~$\building$.
\end{enumerate}
\end{definition}

We denote by~$\nested(\btree) \eqdef \set{\descendants(v, \btree)}{v \text{ vertex of } \btree \text{ distinct from its root}}$ the $\building$-nested set corresponding to a $\building$-tree~$\btree$. Note that $\nested(\btree)$ is a maximal $\building$-nested set if and only if all the vertices of~$\btree$ are labeled by singletons of~$\ground$. We then identify a vertex of~$\btree$ with the element of~$\ground$ labeling~it.

The $\building$-nested sets and the $\building$-trees naturally encode a geometric representation of the $\building$-nested complex as a complete simplicial fan. In the next definition, we define~$\one_R \eqdef \sum_{r \in R} e_r$, for~$R \subseteq \ground$, and we denote by~$\bar\one_R$ the projection of~$\one_R$ to~$\HH$ orthogonal to~$\one_\ground$. Moreover, we consider~$\HH$ as a linear space.

\begin{definition}
\label{def:nestedFan}
The \defn{$\building$-nested fan} $\nestedFan(\building)$ is the complete simplicial fan of~$\HH \eqdef \Hyp(\ground)$ with a cone
\[
\normalCone(\nested) \eqdef \cone\set{\bar\one_N}{N \in \nested} = \set{\b{x} \in \HH}{x_r \le x_s \text{ for each path } r \to s \text{ in } \btree} \defeq \normalCone(\btree)
\]
for each $\building$-nested set~$\nested$ and~$\building$-tree~$\btree$ with~$\nested = \nested(\btree)$.
\end{definition}

The $\building$-nested fan is the normal fan of various deformed permutahedra (see Definition~\ref{def:deformedPermutahedron}). We want to underline two relevant examples:
\begin{enumerate}[(i)]
\item the deformed permutahedron~$\Defo(\b{z})$ with right hand side~$\b{z} \eqdef (z_R)_{R \subseteq \ground}$ defined by~${z_R = 3^{|R|-2}}$ if~$R \in \building$ and $z_R = \infty$ otherwise, see~\cite{Devadoss};
\item the Minkowski sum~$\Mink(\one_\building)$ of the faces of the standard simplex corresponding to all the elements of the building set~$\building$, see~\cite[Section~7]{Postnikov}.
\end{enumerate}
However, these two realizations are not always removahedra. In general, the support functions realizing the normal fan~$\nestedFan(\building)$ can be characterized by local conditions~\cite[Proposition~6.3]{Zelevinsky}, but it is difficult to use these conditions to characterize which nested complexes can be realized as removahedra. In this note, we adopt a different approach.

\subsection{Results}

The objective of this note is to discuss necessary and sufficient conditions for the $\building$-nested fan to be the normal fan of a removahedron (see Definition~\ref{def:removahedron}). We thus consider the removahedron~$\Remo(\building)$ described by the facet defining inequalities of the permutahedron~$\Perm(\ground)$ whose normal vectors are rays of the $\building$-nested fan, \ie
\[
\Remo(\building) \eqdef \HH \cap \bigcap_{B \in \building} \HS(B) = \biggset{\b{x} \in \HH}{\sum_{s \in B} x_s \ge \binom{|B|+1}{2} \text{ for all } B \in \building}.
\]

\begin{example}
\label{exm:remo}
Figure~\ref{fig:remo} represents the removahedra~$\Remo(\building\ex{0})$, $\Remo(\building\ex{1})$ and~$\Remo(\building\ex{2})$ corresponding to the graphical building sets~$\building\ex{0}$, $\building\ex{1}$ and~$\building\ex{2}$ of Example~\ref{exm:buildingSets}. Observe that~$\Remo(\building\ex{0})$ and~$\Remo(\building\ex{2})$ realize the corresponding nested complexes, whereas~$\Remo(\building\ex{1})$ is not even a simple polytope.

\begin{figure}[h]
  \capstart
  \centerline{\includegraphics[width=\textwidth]{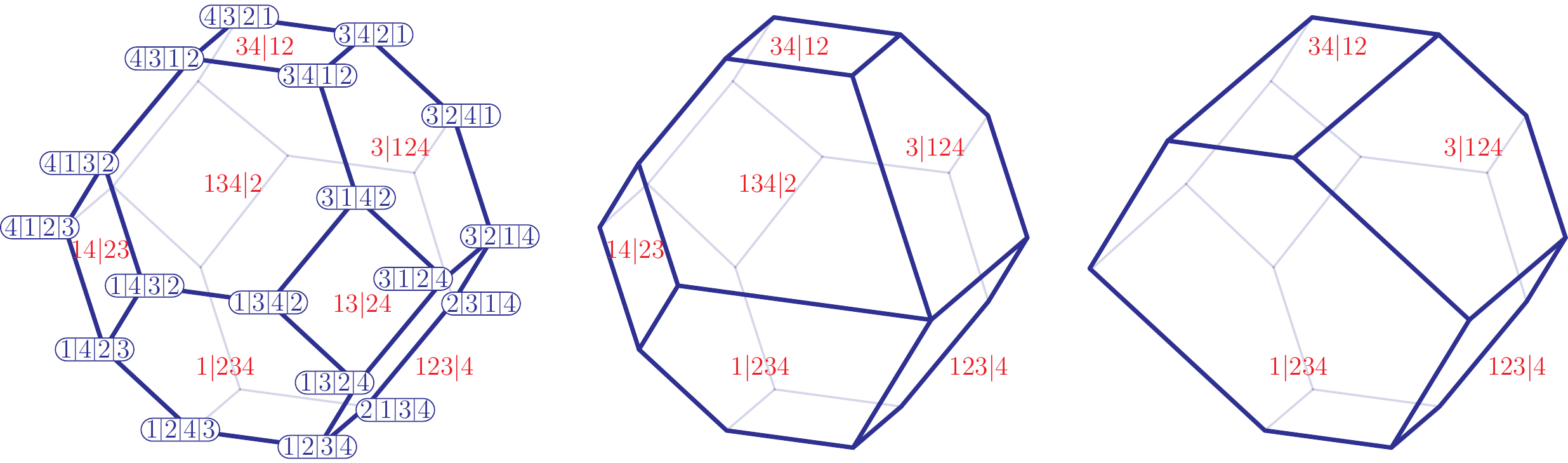}}
  \caption{The $3$-dimensional permutahedron~$\Perm([4]) = \Remo(\building\ex{0})$ and the removahedra~$\Remo(\building\ex{1})$ and~$\Remo(\building\ex{2})$.}
  \label{fig:remo}
\end{figure}
\end{example}

We want to understand when does $\Remo(\building)$ realize the nested complex~$\nestedComplex(\building)$. The following statement provides a general sufficient condition.

\begin{theorem}
\label{theo:main}
If~$\building$ is a connected building set closed under intersection, then the normal fan of the removahedron~$\Remo(\building)$ is the $\building$-nested fan~$\nestedFan(\building)$. In particular,~$\Remo(\building)$ is a simple polytope.
\end{theorem}

We provide two different complementary proofs of this result:
\begin{enumerate}[(i)]
\item In Section~\ref{sec:countingMaximalPaths}, we apply a result from~\cite{HohlwegLangeThomas} which characterizes the valid right hand sides to realize a complete simplicial fan as the normal fan of a convex polytope. For this, we first compute for each maximal $\building$-tree~$\btree$ the intersection point~$\point(\btree)$ of all facet defining hyperplanes of~$\Remo(\building)$ normal to the rays of the cone~$\normalCone(\btree)$. We then show that the vector joining the points~$\point(\btree)$ and~$\point(\btree')$ corresponding to two adjacent cones~$\normalCone(\btree)$ and~$\normalCone(\btree')$ points in the right direction for~$\Remo(\building)$ to realize the nested fan~$\nestedFan(\building)$.
\item In Section~\ref{sec:Minkowski}, we show that certain Minkowski sums of dilated faces of the standard simplex realize the $\building$-nested fan as soon as all $\building$-paths appear as summands. We then find the appropriate dilation factors for the resulting polytope to be a removahedron.
\end{enumerate}

\medskip
Relevant examples of application of Theorem~\ref{theo:main} arise from graphical building sets of chordful graphs. If we restrict to graphical building sets, Lemma~\ref{lem:sufficientGraphicalBuildingSets} shows that chordfulness of~$\graphG$ is also a necessary condition for the $\building\graphG$-nested fan~$\nestedFan(\building\graphG)$ to be the normal fan of~$\Remo(\building\graphG)$. We therefore obtain the following characterization of the graphical building sets whose nested fan is the normal fan of a removahedron. This characterization is illustrated by Example~\ref{exm:remo}.

\begin{theorem}
\label{theo:chordful}
The $\building\graphG$-nested fan~$\nestedFan(\building\graphG)$ is the normal fan of the removahedron~$\Remo(\building\graphG)$ if and only if the graph~$\graphG$ is chordful.
\end{theorem}

\begin{example}
Specific families of chordful graphs provide relevant examples of graph associahedra realized by removahedra, \eg:
\begin{itemize}
\item the path associahedra, \aka classical associahedra~\cite{Loday},
\item the star associahedra, \aka stellohedra~\cite{PostnikovReinerWilliams},
\item the tree associahedra~\cite{Pilaud},
\item the complete graph associahedra, \aka classical permutahedra.
\end{itemize}
We note however that the cycle associahedra, \aka cyclohedra, cannot be realized by removahedra.
\end{example}

To conclude, we observe that a general building set~$\building$ does not need to be closed under intersection for the removahedron~$\Remo(\building)$ to realize the $\building$-nested fan~$\fan(\building)$. In fact, our first proof of Theorem~\ref{theo:main} shows the following refinement. We say that two building blocks~$B, B' \in \building$ are \defn{exchangeable} if there exists two maximal $\building$-nested sets~$\nested, \nested'$ such that~$\nested \ssm \{B\} = \nested' \ssm \{B'\}$.

\begin{theorem}
\label{theo:main2}
If the intersection of any two exchangeable building blocks of~$\building$ also belongs to~$\building$, then the normal fan of the removahedron~$\Remo(\building)$ is the $\building$-nested fan~$\nestedFan(\building)$.
\end{theorem}

\enlargethispage{.1cm}
This result is illustrated by the building set~$\building\ex{3}$ of Example~\ref{exm:buildingSets} and its removahedron~$\Remo(\building\ex{3})$ represented in Figure~\ref{fig:strange}. However, the condition of Theorem~\ref{theo:main2} is still not necessary for the removahedron~$\Remo(\building)$ to realize the nested complex~$\nestedComplex(\building)$. For example, we invite the reader to check that the removahedron~$\Remo(\building\ex{4})$ of the building set~$\building\ex{4}$ of Example~\ref{exm:buildingSets} realizes the corresponding nested complex, even if~$\{1,2,3,4\} \cap \{3,4,5\} = \{3,4\} \notin \building\ex{4}$ while~$\{1,2,3,4\}$ and~$\{3,4,5\}$ are exchangeable. Corollary~\ref{coro:characterization} gives a necessary and sufficient, thought unpractical, condition for the removahedron~$\Remo(\building)$ of an arbitrary building set~$\building$ to realize the nested~fan~$\nestedFan(\building)$.

\begin{figure}[h]
  \capstart
  \centerline{\includegraphics[width=.8\textwidth]{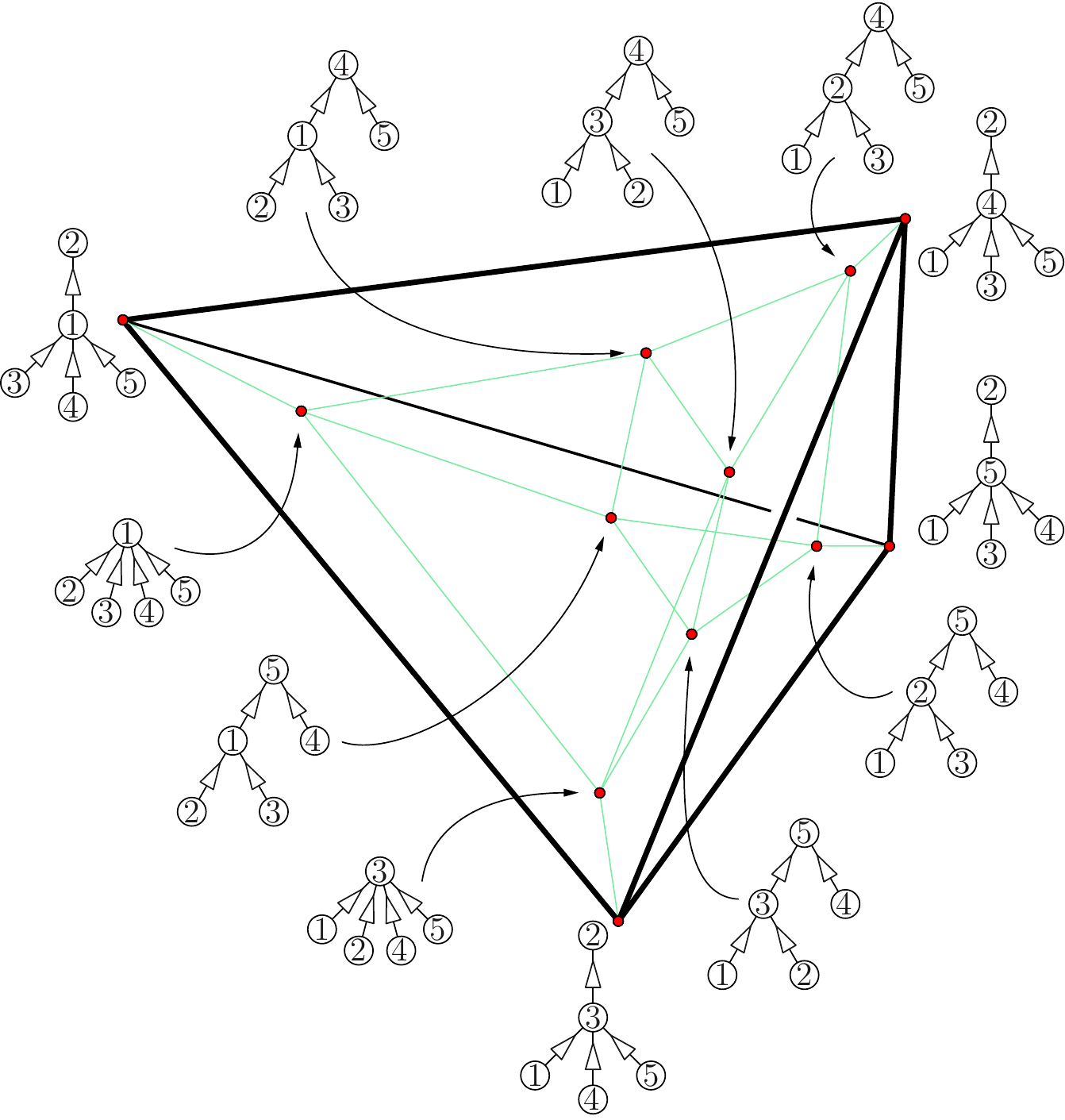}}
  \caption{The $4$-dimensional removahedron~$\Remo(\building\ex{3})$ realizes the $\building\ex{3}$-nested fan, although~$\building\ex{3}$ is not closed under intersection.}
  \label{fig:strange}
\end{figure}


\section{Counting paths in maximal $\building$-trees}
\label{sec:countingMaximalPaths}

Our first approach to Theorem~\ref{theo:main} is the following characterization of the valid right hand sides to realize a complete simplicial fan as the normal fan of a convex polytope. A proof of this statement can be found \eg in~\cite[Theorem~4.1]{HohlwegLangeThomas}.

\begin{theorem}[\protect{\cite[Theorem~4.1]{HohlwegLangeThomas}}]
\label{theo:HohlwegLangeThomas}
Given a complete simplicial fan~$\fan$ in~$\R^d$, consider for each ray~$\rho$ of~$\fan$ a half-space~$\HS_\rho$ of~$\R^d$ containing the origin and defined by a hyperplane~$\Hyp_\rho$ orthogonal to~$\rho$. For each maximal cone~$C$ of~$\fan$, let~$\point(C) \in \R^d$ be the intersection of the hyperplanes~$\Hyp_\rho$ for~$\rho \in C$. Then the following assertions are equivalent:
\begin{enumerate}[(i)]
\item The vector~$\point(C') - \point(C)$ points from~$C$ to~$C'$ for any two adjacent maximal cones~$C$, $C'$ of~$\fan$.
\item The polytopes
\[
\conv\set{\point(C)}{C \text{ maximal cone of } \fan} \quad\text{ and }\quad
\bigcap_{\rho \text{ ray of } \fan} \HS_\rho
\]
coincide and their normal fan is~$\fan$.

\end{enumerate}
\end{theorem}

Since we are given a complete simplicial fan~$\nestedFan(\building)$ and we want to prescribe the right hand sides of the inequalities to describe a polytope realizing it, we are precisely in the situation of Theorem~\ref{theo:HohlwegLangeThomas}. Our first step is to compute the intersection points of the hyperplanes normal to the rays of a maximal cone of~$\nestedFan(\building)$. We associate to any maximal $\building$-tree~$\btree$ a point~$\point(\btree) \in \R^\ground$ whose coordinate~$\point(\btree)_s$ is defined as the number of paths~$\pi$ in~$\btree$ such that~$s$ is the topmost vertex~$\topmost(\pi)$ of~$\pi$ in~$\btree$. Note that all coordinates of~$\point(\btree)$ are strictly positive integers since we always count the trivial path reduced to the vertex~$s$ of~$\btree$. The following lemma ensures that the point~$\point(\btree)$ lies on all hyperplanes of~$\Remo(\building)$ normal to the rays of the cone~$\normalCone(\btree)$.

\begin{lemma}
\label{lem:intersectionPoint}
For any maximal $\building$-tree~$\btree$ and any element~$s \in \ground$, the point~$\point(\btree)$ lies on the hyperplane~$\Hyp(\descendants(s, \btree))$.
\end{lemma}

The proof of this lemma is inspired from similar statements in~\cite[Proposition~6]{LangePilaud} and~\cite[Proposition~60]{Pilaud}. Although the latter covers the present result, we provide a simpler and self-contained proof for the convenience of the reader.

\begin{proof}[Proof of Lemma~\ref{lem:intersectionPoint}]
Consider a $\building$-tree~$\btree$, and let~$\Pi$ be the set of all paths in~$\btree$. For any~$\pi \in \Pi$, the topmost vertex~$\topmost(\pi)$ of~$\pi$ in~$\btree$ is a descendant of~$s$ in~$\btree$ if and only if both endpoints of~$\pi$ are descendants of~$s$ in~$\btree$. It follows that 
\[
\sum_{r \in \descendants(s, \btree)} \point(\btree)_r = \sum_{r \in \descendants(s, \btree)} \sum_{\pi \in \Pi} \one_{\topmost(\pi) \, = \, r} = \sum_{\pi \in \Pi} \one_{\topmost(\pi) \, \in \, \descendants(s, \btree)} = \binom{|\descendants(s, \btree)|+1}{2}
\]
since the number of paths~$\pi \in \Pi$ such that~$\topmost(\pi) \in \descendants(s, \btree)$ is just the number of pairs of endpoints in~$\descendants(s, \btree)$, with possible repetition. We therefore have~$\point(\btree) \in \Hyp(\descendants(s, \btree))$.
\end{proof}

Guided by Theorem~\ref{theo:HohlwegLangeThomas}, we now compute the difference~$\point(\btree') - \point(\btree)$ for two adjacent maximal $\building$-trees~$\btree$ and~$\btree'$. Let~$s,s' \in \ground$ be such that the cones~$\normalCone(\btree)$ and~$\normalCone(\btree')$ are separated by the hyperplane of equation~$x_s = x_{s'}$, and moreover~$x_s \le x_{s'}$ in~$\normalCone(\btree)$ while~$x_s \ge x_{s'}$ in~$\normalCone(\btree')$. Let~$\bar\btree$ denote the tree obtained by contracting the arc~$s \to s'$ in~$\btree$ or, equivalently, the arc~$s' \to s$ in~$\btree'$. Since both~$\btree$ and~$\btree'$ contract to~$\bar\btree$, the children of the node of~$\bar\btree$ labeled by~$\{s,s'\}$ are all children of~$s$ or~$s'$ in both~$\btree$ and~$\btree'$. We denote by~$S$ (resp.~by~$S'$) the elements of~$\ground$ which are children of~$s$ (resp.~of~$s'$) in both~$\btree$ and~$\btree'$. In contrast, we let~$R$ denote the elements of~$\ground$ which are children of~$s$ in~$\btree$ and of~$s'$ in~$\btree'$, and~$R'$ those which are children of~$s'$ in~$\btree$ and of~$s$ in~$\btree'$. These notations are summarized on \fref{fig:flip}. For~$r \in S \cup S' \cup R \cup R'$, we denote the set of descendants of~$r$ by~$\descendants(r) \eqdef \descendants(r, \btree) = \descendants(r, \btree') = \descendants(r, \bar\btree)$. 

\begin{figure}[h]
  \capstart
  \centerline{\includegraphics[scale=1]{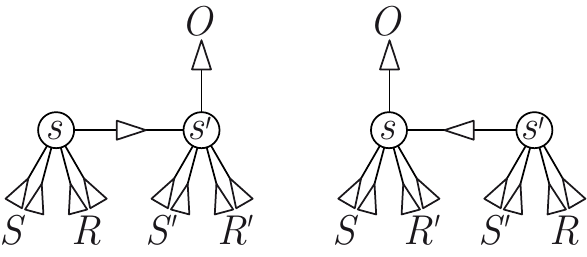}}
  \caption{Two adjacent maximal $\building$-trees~$\btree$ (left) and~$\btree'$ (right).}
  \label{fig:flip}
\end{figure}

\begin{lemma}
\label{lem:diff}
Let~$\building$ be a building set and $\btree$, $\btree'$ be two adjacent maximal $\building$-trees. Using the notation just introduced,  we set
\[
\delta_X \eqdef \sum_{x \in X} |\descendants(x)|
\qquad\text{and}\qquad
\pi_X \eqdef \sum_{x \ne x' \in X} |\descendants(x)| \cdot |\descendants(x')|
\]
for~$X \in \{S,S',R,R'\}$. Then the difference~$\point(\btree') - \point(\btree)$ is given by
\[
\point(\btree') - \point(\btree) = \Delta(\btree, \btree') \cdot (e_s - e_{s'}),
\]
where the coefficient~$\Delta(\btree, \btree')$ is defined by
\[
\Delta(\btree, \btree') \eqdef (\delta_S+1)(\delta_{S'}+1) +\delta_{R'}(\delta_S+\delta_{S'}+\delta_R+2) + \pi_{R'} - \pi_R.
\]
\end{lemma}

\begin{proof}
By definition of the coordinates of~$\point(\btree)$, we compute
\begin{align*}
\point(\btree)_s &= (\delta_S+1)(\delta_R+1) + \pi_S + \pi_R,\\
\point(\btree)_{s'} &= (\delta_{S'}+1)(\delta_{R'}+1) + (\delta_{S'}+\delta_{R'})(\delta_{R}+\delta_{S}+1) + 1 + \delta_S + \delta_R + \pi_{S'} + \pi_{R'},
\end{align*}
and
\begin{align*}
\point(\btree')_s &= (\delta_{S}+1)(\delta_{R'}+1) + (\delta_{S}+\delta_{R'})(\delta_{R}+\delta_{S'}+1) + 1 + \delta_{S'} + \delta_R + \pi_{S} + \pi_{R'},\\
\point(\btree')_{s'} &= (\delta_{S'}+1)(\delta_R+1) + \pi_{S'} + \pi_R.
\end{align*}
Moreover, the coordinates~$\point(\btree)_r$ and~$\point(\btree')_r$ coincide if~$r \in \ground \ssm \{s,s'\}$ since the flip from~$\btree$ to~$\btree'$ did not affect the children of the node~$r$. The result immediately follows.
\end{proof}

Combining Theorem~\ref{theo:HohlwegLangeThomas} and Lemmas~\ref{lem:intersectionPoint} and~\ref{lem:diff}, we thus obtain the following characterization.

\begin{corollary}
\label{coro:characterization}
The $\building$-nested fan~$\nestedFan(\building)$ is the normal fan of the removahedron~$\Remo(\building)$ if and only if $\Delta(\btree, \btree') > 0$ for any pair of adjacent maximal $\building$-trees~$\btree, \btree'$.
\end{corollary}

The following lemma gives a sufficient condition for this property to hold.

\begin{lemma}
\label{lem:IntersectionImpliesRemovahedra}
For any two adjacent maximal $\building$-trees~$\btree, \btree'$ as in Lemma~\ref{lem:diff}, if~$\descendants(s, \btree) \cap \descendants(s', \btree')$ belongs to~$\building \cup \{\varnothing\}$, then~$\Delta(\btree, \btree') > 0$.
\end{lemma}

\begin{proof}
By assumption, the set
\[
\bigcup_{r \in R} \descendants(r) = \descendants(s, \btree) \cap \descendants(s', \btree')
\]
either belongs to the building set~$\building$ or is empty. Since $\descendants(r) \cap \descendants(r') = \varnothing$ for $r \ne r' \in R$ and $\descendants(r) \in \building$ for $r \in R$, we conclude that $|R| \leq 1$ by Condition~\textit{(N2)} in Definition~\ref{def:nested}. Thus~$\pi_R = 0$, $\delta_S \ge 0$ and~$\delta_{S'} \ge 0$. The statement follows.
\end{proof}

It follows that if~$\building$ is closed under intersection of exchangeable elements, then the $\building$-nested fan~$\nestedFan(\building)$ is the normal fan of the removahedron~$\Remo(\building)$. This concludes our first proof of Theorems~\ref{theo:main} and~\ref{theo:main2}. For graphical building sets, we conclude from Lemma~\ref{lem:graphicalBuildingSet} that~$\Remo(\building\graphG)$ realizes~$\nestedFan(\building)$ as soon as $\graphG$ is chordful. Conversely, the following lemma shows that the condition of Corollary~\ref{coro:characterization} is never satisfied for building sets of non chordful graphs, thus concluding the proof of the characterization of Theorem~\ref{theo:chordful}.

\begin{lemma}
\label{lem:sufficientGraphicalBuildingSets}
Let~$\graphG$ be a connected graph that is not chordful. Then there exist two adjacent maximal $\building\graphG$-trees~$\btree, \btree'$ such that~$\Delta(\btree, \btree') \le 0$
\end{lemma}

\begin{proof}
Consider a cycle~$\cycle$ in~$\graphG$ not inducing a clique and choose two vertices~$a,b \in \cycle$ not connected by a chord. As $a$ and~$b$ are not connected, $\{\{a\},\{b\}\}$ is a $\building$-nested set. We complete it to a $\building$-nested set~$\bar\nested$ formed by subsets of~$\cycle$ all containing either~$a$ or~$b$, such that~$\bar\nested$ be maximal for this property. Let~$B_a$ and~$B_b$ denote the maximal elements of~$\bar\nested$ containing~$a$ and~$b$, respectively. By maximality of~$\bar\nested$, all remaining vertices in~$\cycle \ssm (B_a \cup B_b)$ are connected to both~$B_a$ and~$B_b$. Moreover, there are at least two such vertices $s, s'$. Consider two maximal $\building$-nested sets~$\nested$ and~$\nested'$ both containing~$\bar\nested$ and~$B_a \cup B_b \cup \{s,s'\}$, and such that~$\nested$ contains $B_a \cup B_b \cup \{s\}$ and~$\nested'$ contains $B_a \cup B_b \cup \{s'\}$. The corresponding $\building$-trees $\btree$ and $\btree'$ are such that node $s'$ covers~$s$ in~$\btree$ and $s$ covers $s'$ in $\btree'$. Moreover, using the notations introduced earlier in this section, $R' = S = S' = \varnothing$ while~$|R| \ge 2$ as $R$~contains one vertex of~$B_a$ and one of~$B_b$. Thus~$\delta_S = \delta_{S'} = \delta_{R'} = \pi_{R'} = 0$ and~$\pi_R \ge 1$, which implies~$\Delta(\btree, \btree') \leq 0$.
\end{proof}

As already observed earlier, for general building sets, the condition of Lemma~\ref{lem:IntersectionImpliesRemovahedra} is not necessary for~$\Remo(\building\graphG)$ to realize~$\nestedFan(\building)$. For example, in the building set~$\building\ex{4}$ of Example~\ref{exm:buildingSets}, the building blocks~$\{1,2,3,4\}$ and~$\{3,4,5\}$ are exchangeable but~$\{1,2,3,4\} \cap \{3,4,5\} = \{3,4\} \notin \building\ex{4}$. However, $\{1,2,3,4\}$ and~$\{3,4,5\}$ are the only two intersecting exchangeable building blocks of~$\building\ex{4}$, and for any two maximal $\building\ex{4}$-trees~$\btree, \btree$ such that~$\nested(\btree) \ssm \big\{ \{1,2,3,4\} \} = \nested(\btree') \ssm \big\{ \{3,4,5\} \}$, we have~${\Delta(\btree, \btree') = 1}$. Therefore, Corollary~\ref{coro:characterization} ensures that~$\Remo(\building\ex{4})$ realizes~$\nestedFan(\building\ex{4})$.

\begin{remark}
The arguments of this first proof of Theorems~\ref{theo:main} and~\ref{theo:main2} can be used to show that any nestohedron can be realized as a skew removahedra. A \defn{skew permutohedron} is the convex hull~$\Perm^\b{p}(\ground) \eqdef \conv \set{\sum_{s \in \ground} p_{\sigma(s)}e_s}{\sigma \in \fS_\ground}$ of the orbit of a generic point~$\b{p} \in \R^\ground$ (\ie $p_s \ne p_s'$ for~$s \ne s' \in \ground$) under the action of the symmetric group~$\fS_\ground$ on~$\R^\ground$ by permutation of coordinates. Equivalently, a skew permutahedron is the deformed permutahedron~$\Perm^\phi(\ground) \eqdef \Defo(\b{z}^\phi)$ for a right hand side~$\b{z}^\phi \eqdef (z^\phi_R)_{R \subseteq \ground} \in \R^{2^\ground}$ defined by~$z^\phi_R \eqdef \phi(|R|)$ for some function~$\phi : \N \to \R_{>0}$.
For example, the classical permutahedron~$\Perm(\ground)$ is the permutahedron~$\Perm^\phi(\ground)$ for~$\phi(n) = \binom{n+1}{2}$. A \defn{skew removahedron} is a polytope
\[
\Remo^\b{p}(\building) = \Remo^\phi(\building) \eqdef \biggset{\b{x} \in \HH}{\sum_{s \in B} x_s \ge \phi(|B|) \text{ for all } B \in \building},
\]
obtained by removing inequalities from the facet description of a skew permutahedron~$\Perm^\b{p}(\ground) = \Perm^\phi(\ground)$. Note that even if skew removahedra have much more freedom than classical removahedra, they do not contain all not all deformed permutahedra.

A consequence of the realization of~\cite{Devadoss} is that all graph associahedra can be realized as skew removahedra, namely by removing facets of the skew permutahedron~$\Perm^\phi(\ground)$ for~$\phi(n) = \gamma^n$ with~$\gamma > 2$. The arguments presented in this Section provide an alternative proof of this result and extend it to all nestohedra. Let us quickly give the proof here.

For a $\building$-tree~$\btree$, consider the point~$\point^\gamma(\btree) \in \R^\ground$ defined by
\[
\sum_{r \in \descendants(s, \btree)} \point^\gamma(\btree)_r = \gamma^{|\descendants(s, \btree)|} \qquad \text{for all } s \in \ground.
\]
We do not need to compute explicitly the coordinates of~$\point^\gamma(\btree)$. We will only use that for any~$s \in \ground$,
\[
\point^\gamma(\btree)_s = \gamma^{|\descendants(s, \btree)|} - \sum_{r \in \descendants(s, \btree)} \gamma^{|\descendants(r, \btree)|}.
\]
Consider now two adjacent maximal $\building$-trees~$\btree, \btree'$ with the same notations as in \fref{fig:flip}. For a subset~$X \in \{S,S',R,R'\}$, define
\[
\delta_X \eqdef \sum_{x \in X} |\descendants(x)|
\qquad\text{and}\qquad
\Gamma_X \eqdef \sum_{x \in X} \gamma^{|\descendants(x)|}.
\]
Using these notations, we compute:
\begin{align*}
\point^\gamma(\btree)_s & = \gamma^{1 + \delta_S + \delta_R} - \Gamma_S - \Gamma_R, \\
\point^\gamma(\btree)_{s'} & = \gamma^{2 + \delta_S + \delta_{R'} + \delta_{S'} + \delta_{R}} - \gamma^{1 + \delta_S + \delta_R} - \Gamma_{S'} - \Gamma_{R'},
\end{align*}
and
\begin{align*}
\point^\gamma(\btree')_s & = \gamma^{2 + \delta_S + \delta_{R'} + \delta_{S'} + \delta_{R}} - \Gamma_S - \Gamma_{R'} - \gamma^{1 + \delta_{S'} + \delta_R}, \\
\point^\gamma(\btree')_{s'} & = \gamma^{1 + \delta_{S'} + \delta_R} - \Gamma_{S'} - \Gamma_R.
\end{align*}
Moreover, the coordinates~$\point(\btree)_r$ and~$\point(\btree')_r$ coincide if~$r \in \ground \ssm \{s,s'\}$ since the flip from~$\btree$ to~$\btree'$ did not affect the children of the node~$r$. We therefore obtain that~$\point^\gamma(\btree') - \point^\gamma(\btree) = \Delta(\btree, \btree') \cdot (e_s - e_{s'})$, where
\begin{align*}
\Delta(\btree, \btree')
& = \gamma^{2 + \delta_S + \delta_{R'} + \delta_{S'} + \delta_R} - \gamma^{1 + \delta_{S'} + \delta_R} - \gamma^{1 + \delta_S + \delta_R} + \Gamma_R - \Gamma_{R'} \\
& \ge \gamma^{\delta_R} \big( \big( \gamma^{1+\delta_S} - 1 \big) \big( \gamma^{1+\delta_{S'}} - 1\big) - 1 \big) + \gamma^{2 + \delta_S + \delta_{S'} + \delta_R} (\gamma^{\delta_{R'}} - 1) - \Gamma_{R'}.
\end{align*}
Since~$\gamma > 2$, we have~$\big( \gamma^{1+\delta_S} - 1 \big) \big( \gamma^{1+\delta_{S'}} - 1\big) - 1 > 0$ and $\gamma^{2 + \delta_S + \delta_{S'} + \delta_R} (\gamma^{\delta_{R'}} - 1) - \Gamma_{R'} \ge \gamma^{\delta_{R'} + 1} - \gamma - \Gamma_{R'} \ge 0$ since~$\gamma^{n+m} \ge \gamma^n + \gamma^m$ for~$1 \le n, m$. We therefore obtain that~$\Delta(\btree, \btree') > 0$ and we conclude by Theorem~\ref{theo:HohlwegLangeThomas}.
\end{remark}


\section{Minkowski sums}
\label{sec:Minkowski}

In this section, we provide an alternative proof that any building set closed under intersection can be realized as a removahedron. The approach of this section is complementary to the previous one since it focusses on Minkowski sums. As illustrated by the following statement observed independently by A.~Postnikov~\cite[Section~7]{Postnikov} and E.-M.~Feichtner and B.~Sturmfels~\cite{FeichtnerSturmfels}, Minkowski sums provide a powerful tool to realize nested complexes.

\begin{theorem}[\protect{\cite[Section~7]{Postnikov}, \cite{FeichtnerSturmfels}}]
\label{theo:Postnikov}
For any building set~$\building$, the $\building$-nested fan~$\fan(\building)$ is the normal fan of the Minkowski sum
\[
\Mink[\building] = \sum_{B \in \building} y_B\simplex_B,
\]
where~$(y_B)_{B \in \building}$ are arbitrary strictly positive real numbers, and~$\simplex_B$ denotes the face of the standard simplex corresponding to~$B$.
\end{theorem}

Remember from Remark~\ref{rem:genericMinkowskiSum} that the normal fan, and thus the combinatorics, of the Minkowski sum~$\Mink[\building]$ only depends on~$\building$, not on the values of the dilation factors~$(y_B)_{B \in \building}$ (as soon as all these values are strictly positive). These Minkowski sums~$\Mink[\building]$ are deformed permutahedra, but not necessarily removahedra. In this section, we relax Theorem~\ref{theo:Postnikov} to give a sufficient condition for a subsum of~$\Mink[\building]$ to keep the same normal fan, and prove that a well-chosen subsum of~$\Mink[\building]$ is indeed a removahedron.

\subsection{Generating sets and building paths}

We say that a subset~$\summands$ of~$\building$ is \defn{generating} if for each~$B \in \building$ and for each~$b \in B$, the set~$B$ is the union of the sets~$C \in \summands$ such that~$b \in C \subseteq B$. The following statement can be seen as a relaxation of Theorem~\ref{theo:Postnikov}: it shows that the Minkowski sum~$\Mink[\summands]$ of the faces of the standard simplex over a generating subset~$\summands$ of~$\building$ still has the same normal fan as the Minkowski sum~$\Mink[\building]$ itself. Note that we do not make here any particular assumption on the building set~$\building$. The proof, adapted from that of~\cite[Theorem 7.4]{Postnikov}, is delayed to the next section.

\begin{theorem}
\label{theo:newNestohedron}
If~$\summands$ is a generating subset of a connected building set~$\building$, then the $\building$-nested fan~$\fan(\building)$ is the normal fan of the Minkowski sum
\[
\Mink[\summands] = \sum_{C \in \summands} y_C \simplex_C,
\]
where~$(y_C)_{C \in \summands}$ are arbitrary strictly positive real numbers.
\end{theorem}

\begin{example}
Given a connected graph~$\graphG$, the set~$\summands\graphG$ of all vertex sets of induced subpaths of~$\graphG$ is a generating subset of the graphical building set~$\building\graphG$. Set here~$y_C = 1$ for all~$C \in \summands\graphG$. Then
\begin{enumerate}
\item if~$\graphG = \pathG$ is a path, then $\building\pathG$ and~$\summands\pathG$ coincide and~$\Mink[\building\pathG] = \Mink[\summands\pathG]$ is precisely Loday's associahedron~\cite{Loday};
\item if~$\graphG = \tree$ is a tree, then~$\Mink[\summands\tree]$ is the (unsigned) tree associahedron of~\cite{Pilaud};
\item if~$\graphG = K_\ground$ is the complete graph, then~$\building K_\ground = 2^\ground \ssm \{\varnothing\}$ while~$\summands K_\ground = \set{R \subseteq \ground}{1 \le |R| \le 2}$. Thus, $\Mink[\summands K_\ground]$ is the classical permutahedron, while~$\Mink[\building K_\ground]$ is a dilated copy of it.
\end{enumerate}
\end{example}

In fact, the notion of paths can be extended from these graphical examples to the more general setting of connected building sets closed under intersection. Consider a building set~$\building$ on the ground set~$\ground$, closed under intersection. For any~$R \subseteq \ground$, we define the \defn{$\building$-hull} of~$R$ to be the smallest element of~$\building$ containing~$R$ (it exists since~$\building$ is closed under intersection). In particular, for any~$s,t \in \ground$, we defined the \defn{$\building$-path} $\bpath(s,t)$ to be the $\building$-hull of~$\{s,t\}$. We denote by~$\bpaths(\building)$ the set of all $\building$-paths.

\begin{example}
For a graphical building set~$\building\graphG$, the $\building\graphG$-paths are precisely the induced subpaths of~$\graphG$, \ie with our notations~$\bpaths(\building\graphG) = \summands\graphG$.
\end{example}

\begin{lemma}
\label{lem:generating}
The set~$\bpaths(\building)$ of all $\building$-paths is a generating subset of~$\building$.
\end{lemma}

\begin{proof}
Let~$B \in \building$ and~$b \in B$. For any~$b' \in B$, the path~$\bpath(b,b')$ contains~$b$ and is contained in~$B$ (indeed, $B$ contains both~$b$ and~$b'$ and thus~$\bpath(b,b')$ by minimality of the latter). Therefore, $b'$ belongs to the union of the sets~$C \in \bpaths(\building)$ such that~$b \in C \subseteq B$. The lemma follows by definition of generating subsets.
\end{proof}

In fact, the set of paths~$\bpaths(\building)$ is the minimal generating subset of~$\building$, in the following sense.

\begin{lemma}
Any generating subset of~$\building$ contains~$\bpaths(\building)$.
\end{lemma}

\begin{proof}
Consider a generating subset~$\summands$ of~$\building$, and~$s,t \in \ground$. The $\building$-path~$\bpath(s,t)$ is the smallest building block containing~$\{s,t\}$. Therefore, $\bpath(s,t)$ has to be in~$\summands$, since otherwise~$t$ would not belong to the union of the sets~$C$ such that~$s \in C \subseteq \bpath(s,t)$.
\end{proof}

From Theorem~\ref{theo:newNestohedron} and Lemma~\ref{lem:generating}, we obtain that any Minkowski sum~$\Mink[\bpaths(\building)]$ is a realization of the $\building$-nested complex. We now have to choose properly the dilation coefficients to obtain a removahedron. For~$S \subseteq \ground$, define the coefficient
\[
\bar y_S \eqdef \big| \bigset{s,t \in \ground^2}{\bpath(s,t) = S} \big|.
\]

\begin{lemma}
\label{lem:dilationCoeff}
The dilation coefficients~$(\bar y_B)_{B \in \building}$ satisfy the following properties:
\begin{enumerate}[(i)]
\item $\bar y_S > 0$ for all~$S \in \bpaths(\building)$, and~$\bar y_S = 0$ otherwise.
\item For all~$B \in \building$,
\[
\sum_{S \subseteq B} \bar y_S = \binom{|B|+1}{2}.
\]
\end{enumerate}
\end{lemma}

\begin{proof}
Point~(i) is clear by definition of the coefficients~$\bar y_S$. We prove Point~(ii) by double counting: for any~$B \in \building$ and any~$b,b' \in B$ (distinct or not), the path~$\bpath(b,b')$ is included in~$B$. We can therefore group pairs of elements of~$B$ according to their $\building$-paths:
\[
\binom{|B|+1}{2} = |\{b,b' \in B\}| = \sum_{S \subseteq B} |\bigset{b,b' \in B}{\bpath(b,b') = S}| = \sum_{S \subseteq B} \bar y_S. \qedhere
\]
\end{proof}

Combining Theorem~\ref{theo:newNestohedron} with Lemmas~\ref{lem:generating} and~\ref{lem:dilationCoeff}, we obtain an alternative proof of Theorem~\ref{theo:main}.

\begin{corollary}
For a building set~$\building$ closed under intersection, the removahedron~$\Remo(\building)$ coincides with the Minkowski sum~$\sum_{B \in \building} \bar y_B \simplex_B$, and its normal fan is the $\building$-nested fan.
\end{corollary}

\begin{remark}
For a graphical building set~$\building\graphG$ of a chordful graph~$\graphG$, the coefficients~$\bar y_C$ are all equal to~$1$ for all subpaths~$C \in \summands\graphG$. It is not anymore true for arbitrary building sets closed under intersection. For example, consider the building set~$\building\ex{5} \eqdef \big\{ \! \{1\}, \{2\}, \{3\}, \{4\}, \{1,2\}, \{1,2,3\} \! \big\}$, for which we obtain
\[
\Remo(\building\ex{5}) =  2\simplex_{\{1,2,3\}}  + \simplex_{\{1,2\}} + \simplex_{\{1\}} + \simplex_{\{2\}} + \simplex_{\{3\}}.
\]
This Minkowski decomposition is illustrated in \fref{fig:removahedron}.

\begin{figure}[b]
  \centerline{
    \begin{tabular}{cc@{\hspace{-.1cm}}c@{\hspace{-.6cm}}c}
      \begin{tikzpicture}
      	\draw (0cm, 0cm) -- (1cm, 1.732cm) -- (2cm, 1.732cm) -- (3cm, 0cm) -- (0cm, 0cm);
      	\filldraw [black] (0cm, 0cm) node[anchor=north east] {$\left( \begin{smallmatrix} 4\\ 1\\ 1 \end{smallmatrix}\right)$} circle (2pt)
            		  (1cm, 0cm) circle (2pt)
            		  (2cm, 0cm) circle (2pt)
            		  (3cm, 0cm) node[anchor=north west] {$\left( \begin{smallmatrix} 1\\ 4\\ 1 \end{smallmatrix}\right)$} circle (2pt)
            		  (.5cm, 0.5*1.732cm) circle (2pt)
            		  (1.5cm, 0.5*1.732cm) circle (2pt)
            		  (2.5cm, 0.5*1.732cm) circle (2pt)
            		  (1cm, 1.732cm) node[anchor=south east] {$\left( \begin{smallmatrix} 2\\ 1\\ 3 \end{smallmatrix}\right)$} circle (2pt)
            		  (2cm, 1.732cm) node[anchor=south west] {$\left( \begin{smallmatrix} 1\\ 2\\ 3 \end{smallmatrix}\right)$} circle (2pt);
      \end{tikzpicture}
      &
      \begin{tikzpicture}
      	\draw (0cm, 0cm) -- (2cm, 0cm) -- (1cm, 1.732cm) -- (0cm, 0cm);
      	\filldraw [black] (0cm, 0cm) node[anchor=north east] {$\left( \begin{smallmatrix} 2\\ 0\\ 0 \end{smallmatrix}\right)$} circle (2pt)
            		  (1cm, 0cm) circle (2pt)
            		  (2cm, 0cm) node[anchor=north west] {$\left( \begin{smallmatrix} 0\\ 2\\ 0 \end{smallmatrix}\right)$} circle (2pt)
            		  (.5cm, 0.5*1.732cm) circle (2pt)
            		  (1.5cm, 0.5*1.732cm) circle (2pt)
            		  (1cm, 1.732cm) node[anchor=south] {$\left( \begin{smallmatrix} 0\\ 0\\ 2 \end{smallmatrix}\right)$} circle (2pt);
      \end{tikzpicture}
      &
      \begin{tikzpicture}[baseline=-1.6cm]
      	\draw (0cm, 0cm) -- (1cm, 0cm);
      	\filldraw [black] (0cm, 0cm) node[anchor=east] {$\left( \begin{smallmatrix} 1\\ 0\\ 0 \end{smallmatrix}\right)$} circle (2pt)
      	                  (1cm, 0cm) node[anchor=west] {$\left( \begin{smallmatrix} 0\\ 1\\ 0 \end{smallmatrix}\right)$} circle (2pt);
      \end{tikzpicture}
      &
      \begin{tikzpicture}[baseline=-.65cm]
      	\filldraw [black] (1cm, 1cm) node[anchor=east] {$\left( \begin{smallmatrix} 1\\ 1\\ 1 \end{smallmatrix}\right)$} circle (2pt);
      \end{tikzpicture}
      \\
      $\qquad\qquad\qquad \Remo(\building\ex{5}) \qquad\qquad =$
      &
      $2 \, \simplex_{\{1,2,3\}}$
      &
      $+ \quad\quad \simplex_{\{1,2\}} \quad +$
      &
      $\qquad \begin{pmatrix} \simplex_{\{1\}} + \\ \simplex_{\{2\}} + \simplex_{\{3\}} \end{pmatrix}$
    \end{tabular}
  }
  \caption[]{The Minkowski decomposition of the $2$-dimensional removahedron~$\Remo(\building\ex{5})$ into faces of the standard simplex.}
  \label{fig:removahedron}
\end{figure}
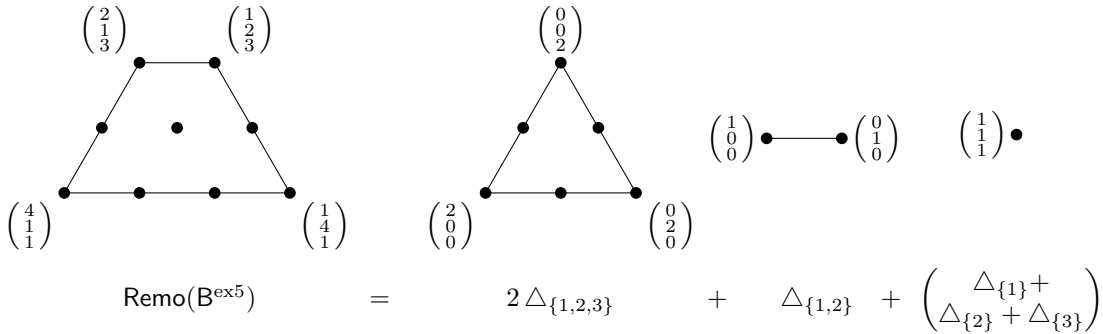
\end{remark}

\subsection{Proof of Theorem~\ref{theo:newNestohedron}}

This section is devoted to the proof of Theorem~\ref{theo:newNestohedron}. We start with the following technical lemma on the affine dimension of Minkowski sums.

\begin{lemma}
\label{lem:dimensionMinkowskiSum}
\begin{enumerate}[(i)]
\item Let~$(P_i)_{i \in I}$ be polytopes lying in orthogonal subspaces of~$\R^n$. Then
\[
\dim \sum P_i = \sum \dim P_i.
\]
\item If~$\I \subseteq 2^\ground$ is such that~$\bigcap \I \neq \varnothing$, then~$\dim \sum_{I \in \I} \simplex_I = |\bigcup \I| - 1$.
\end{enumerate}
\end{lemma}

\begin{proof}
Point~(i) is immediate as the union of bases of the linear spaces generated by the polytopes~$P_i$ is a basis of the linear space generated by~$\sum P_i$. For Point~(ii), fix~$x \in \bigcap \I$ and an arbitrary order~$I_1, \dots, I_p$ on~$\I$. Define~$I'_j \eqdef I_j \ssm \big( \{x\} \cup \bigcup_{k < j} I_k \big)$. We then have
\[
\dim \sum_{I \in \I} \simplex_I \ge \dim \sum_{\substack{j \in [p] \\ I'_j \ne \varnothing}} \simplex_{I'_j \cup \{x\}} \ge \sum_{\substack{j \in [p] \\ I'_j \ne \varnothing}} \dim \simplex_{I'_j \cup \{x\}} = \sum_{\substack{j \in [p] \\ I'_j \ne \varnothing}} |I'_j| = \big| \bigcup \I \big| - 1,
\]
where the first inequality holds since $\simplex_{I'_j \cup \{x\}}$ is a face of~$\simplex_{I_j}$, the second one is a consequence of Point~(i), and the last equality holds since we have the partition
\[
\big( \bigcup \I \big) \ssm \{x\}= \bigsqcup_{\substack{j \in [p] \\ I'_j \ne \varnothing}} I'_j. \qedhere
\]
\end{proof}

\begin{proof}[Proof of Theorem~\ref{theo:newNestohedron}]
Let~$\summands$ be a generating subset of a connected building set~$\building$, let $\b{y} \eqdef (y_C)_{C \in \summands}$ be strictly positive real numbers, let~$\b{z} \eqdef (z_R)_{R \subseteq \ground}$ be defined by~$z_R = \sum_{C \subseteq R} y_C$, and consider the polytope~$\Mink(\b{y}) = \Defo(\b{z})$.

Let~$\btree$ be a $\building$-tree and~$\nested \eqdef \nested(\btree)$ be the corresponding $\building$-nested set. For~$N \in \nested$, let
\[
X_N \eqdef N \ssm \bigcup_{\substack{N' \in \nested \\ N' \subsetneq N}} N'
\]
denote the label corresponding to~$N$ in the $\building$-tree~$\btree$, so that~$(X_N)_{N \in \nested}$ partitions~$\ground$. 

For~$C \in \summands$, we define~$N(C)$ to be the inclusion maximal element~$N$ of~$\nested$ such that~$C \cap X_N \ne \varnothing$. Note that this element is unique: otherwise, the union of the maximal elements $N \in \nested$ such that~$C \cap X_N \ne \varnothing$ would be contained in~$\building$, thus contradicting Condition~\textit{(N2)} in Definition~\ref{def:nested}. Observe also that~$N(C)$ is the inclusion minimal element~$N$ of~$\nested$ such that~$C \subseteq N$.

We now define
\[
F_\nested \eqdef \sum_{C \in \summands} y_C \simplex_{C \cap X_{N(C)}}.
\]
We will show below that~$F_\nested$ is a face of~$\Mink(\b{y})$ whose normal cone is precisely the cone~$\normalCone(\nested)$. The map~$\nested \to F_\nested$ thus defines an anti-isomorphism from the nested complex~$\nestedComplex(\building)$ to the face lattice of~$\Mink(\b{y})$.

Consider any vector~$\b{f} \eqdef (f_s)_{s \in \ground}$ in the relative interior of the cone~$\normalCone(\nested)$. This implies that~${f_s = f_N}$ is constant on each~$N \in \nested$, and $f_N < f_{N'}$ for~$N, N' \in \nested$ with~$N \subsetneq N'$. Let~$f : \R^\ground \to \R$ be the linear functional defined by~$f(\b{x}) = \dotprod{\b{f}}{\b{x}} = \sum_{s \in \ground} f_s x_s$. Since~$N(C)$ is the inclusion maximal element~$N$ of~$\nested$ such that~$C \cap N \ne \varnothing$ and~$N \to f_N$ is increasing, the face of~$\simplex_C$ maximizing~$f$ is precisely~$\simplex_{C \cap N(C)}$. It follows that~$F_\nested$ is the face maximizing~$f$ on~$\Mink(\b{y})$, since the face maximizing~$f$ on a Minkowski sum is the Minkowski sum of the faces maximizing~$f$ on each summand. We conclude that~$F_\nested$ is a face of~$\Mink(\b{y})$ whose normal cone contains at least~$\normalCone(\nested)$, and therefore that the map~$\nested \to F_\nested$ is a poset anti-homomorphism.

To conclude, it is now sufficient to prove that the dimension of~$F_\nested$ is indeed~$|\ground| - |\nested|$. The inequality~$\dim F_\nested \le |\ground| - |\nested|$ is clear since the normal cone of~$F_\nested$ contains the cone~$\normalCone(\nested)$. To obtain the reverse inequality, observe that
\begin{align*}
\dim F_\nested & = \dim \bigg( \sum_{C \in \summands} y_C \simplex_{C \cap X_{N(C)}} \bigg)
= \dim \bigg( \sum_{N \in \nested} \sum_{\substack{C \in \summands \\ N(C) = N}} \simplex_{C \cap X_N} \bigg) \\
& \ge \sum_{N \in \nested} \dim \bigg( \sum_{\substack{C \in \summands \\ N(C) = N}} \simplex_{C \cap X_N} \bigg)
\ge \sum_{N \in \nested} (|X_N| - 1) = |\ground| - |\nested|.
\end{align*}
The first inequality holds by Lemma~\ref{lem:dimensionMinkowskiSum}\,(i) since the $X_N$ are disjoints. The second inequality follows from the assumption that~$\summands$ is a generating subset of~$\building$. Indeed, fix~$N \in \nested$ and pick an element~$x \in X_N$. Observe that if~$C \in \summands$ is such that~$x \in C \subseteq N$, then~$N(C) = N$. Moreover, since $N$ is the union of the elements~$C \in \summands$ such that $x \in C \subseteq N$, we obtain that $X_N$ is the union of the sets~$C \cap X_N$ over the elements~$C \in \summands$ such that $x \in C \subseteq N$. By Lemma~\ref{lem:dimensionMinkowskiSum}\,(ii), this implies that
\[
\dim \bigg( \sum_{\substack{C \in \summands \\ N(C) = N}} \simplex_{C \cap X_N} \bigg) \ge \dim \bigg( \sum_{\substack{C \in \summands \\ x \in C \subseteq N}} \simplex_{C \cap X_N} \bigg) \ge |X_N| - 1.
\]
This concludes the proof that the dimension of~$F_\nested$ is given by~$|\ground| - |\nested|$, and thus that the map~${\nested \to F_\nested}$ is an anti-isomorphism.
\end{proof}


\section*{Acknoledgments}

I thank Carsten Lange for helpful discussions on the content and presentation of this paper.


\bibliographystyle{alpha}
\bibliography{removahedra}

\begin{thebibliography}{PRW08}

\bibitem[ABD10]{ArdilaBenedettiDoker}
Federico Ardila, Carolina Benedetti, and Jeffrey Doker.
\newblock Matroid polytopes and their volumes.
\newblock {\em Discrete Comput.~Geom.}, 43(4):841--854, 2010.

\bibitem[CD06]{CarrDevadoss}
Michael~P. Carr and Satyan~L. Devadoss.
\newblock Coxeter complexes and graph-associahedra.
\newblock {\em Topology Appl.}, 153(12):2155--2168, 2006.

\bibitem[Dev09]{Devadoss}
Satyan~L. Devadoss.
\newblock A realization of graph associahedra.
\newblock {\em Discrete Math.}, 309(1):271--276, 2009.

\bibitem[FS05]{FeichtnerSturmfels}
Eva~Maria Feichtner and Bernd Sturmfels.
\newblock Matroid polytopes, nested sets and {B}ergman fans.
\newblock {\em Port. Math. (N.S.)}, 62(4):437--468, 2005.

\bibitem[HL07]{HohlwegLange}
Christophe Hohlweg and Carsten Lange.
\newblock Realizations of the associahedron and cyclohedron.
\newblock {\em Discrete Comput.~Geom.}, 37(4):517--543, 2007.

\bibitem[HLT11]{HohlwegLangeThomas}
Christophe Hohlweg, Carsten Lange, and Hugh Thomas.
\newblock Permutahedra and generalized associahedra.
\newblock {\em Adv. Math.}, 226(1):608--640, 2011.

\bibitem[Lod04]{Loday}
Jean-Louis Loday.
\newblock Realization of the {S}tasheff polytope.
\newblock {\em Arch.~Math.~(Basel)}, 83(3):267--278, 2004.

\bibitem[LP13]{LangePilaud}
Carsten Lange and Vincent Pilaud.
\newblock Using spines to revisit a construction of the associahedron.
\newblock Preprint, \texttt{arXiv:1307.4391}, 2013.

\bibitem[Pil13]{Pilaud}
Vincent Pilaud.
\newblock Signed tree associahedra.
\newblock Preprint, \texttt{arXiv:1309.5222}, 2013.

\bibitem[Pos09]{Postnikov}
Alexander Postnikov.
\newblock Permutohedra, associahedra, and beyond.
\newblock {\em Int. Math. Res. Not. IMRN}, (6):1026--1106, 2009.

\bibitem[PRW08]{PostnikovReinerWilliams}
Alexander Postnikov, Victor Reiner, and Lauren~K. Williams.
\newblock Faces of generalized permutohedra.
\newblock {\em Doc.~Math.}, 13:207--273, 2008.

\bibitem[Zel06]{Zelevinsky}
Andrei Zelevinsky.
\newblock Nested complexes and their polyhedral realizations.
\newblock {\em Pure Appl. Math. Q.}, 2(3):655--671, 2006.

\bibitem[Zie95]{Ziegler}
G{\"u}nter~M. Ziegler.
\newblock {\em Lectures on polytopes}, volume 152 of {\em Graduate Texts in
  Mathematics}.
\newblock Springer-Verlag, New York, 1995.

\end{thebibliography}
\label{sec:biblio}

\end{document}